\documentclass{elsarticle}
\usepackage[margin=2.5cm]{geometry}%
\usepackage{pifont}
\usepackage{natbib}
  \biboptions{sort&compress}
\usepackage{graphicx}
\usepackage{epstopdf} 
	\DeclareGraphicsExtensions{.pdf,.png,.jpg,.eps,.ps}

\usepackage{hyperref}

\usepackage{amsmath}		
\usepackage{amssymb}
\usepackage{amsthm}

\usepackage[usenames,dvipsnames]{color}
\usepackage{algorithm}
\usepackage{algorithmicx}
\usepackage{algpseudocode}
\usepackage{pstool} 
\usepackage{psfrag}

\usepackage{comment}
\usepackage{caption}
\usepackage{subcaption}

\newcommand{\subs}[1]{\mathcal{#1}}

\newcommand{\mset}{\mathbf{M}}
\newcommand{\graph}{\mathbf{G}}
\newcommand{\reels}{\mathbb{R}} 
 
\newcommand{\tran}{\top\!}
\newcommand{\jsr}{\hat \rho}
\newcommand{\card}[1]{|{#1}|}
\newcommand{\spans}{\text{span}}
\newcommand{\kron}{\otimes}
\renewcommand{\ker}[1]{\text{Ker}\left (#1 \right)}
\newcommand{\im}[1]{\text{Im}(#1)}

\newcommand{\bigO}{\mathcal{O}}
\newtheorem{theorem}{Theorem}[section]
\newtheorem{definition}{Definition}
\newtheorem{remark}{Remark}
\newtheorem{lemma}[theorem]{Lemma}
\newtheorem{proposition}[theorem]{Proposition}

\newtheorem{example}{Example}

\begin{document}

\begin{frontmatter}

\title{Deciding the boundedness and dead-beat stability of\\ constrained switching systems.}
\author{Matthew Philippe \fnref{uclthx}}
\ead{matthew.philippe@uclouvain.be}
\author{Gilles Millerioux \fnref{nancythx}}
\ead{gilles.millerioux@univ-lorraine.fr}
\author{Rapha\"el M. Jungers \fnref{uclthx}}
\ead{raphael.jungers@uclouvain.be}
\fntext[uclthx]{M. Philippe (corresponding author) and R. M. Jungers are with the ICTEAM Institute, Universit\'e Catholique de Louvain.
M. Philippe is a F.N.R.S./F.R.I.A. fellow; R. Jungers  is a F.R.S./F.N.R.S. research associate.
They are supported by the Belgian Interuniversity Attraction Poles, and by the ARC grant 13/18-054 (Communaut\'{e} fran\c{c}aise de Belgique). }
\fntext[nancythx]{G. Millerioux is with the 
Research Center for Automatic Control of Nancy (CRAN),
Universit\'e de Lorraine. This work was partially supported by Research Grants ANR-13-INSE-0005-01 from the Agence Nationale de la Recherche in France.
}
\begin{abstract} 
We study computational questions related with the stability of discrete-time linear switching systems with switching sequences constrained by an automaton.\\
We first present a decidable sufficient condition for their boundedness when the maximal exponential growth rate equals one.
The condition generalizes the notion of the irreducibility of a matrix set, which is a well known sufficient condition for
 boundedness in the arbitrary switching (i.e. unconstrained) case.\\
Second, we provide a polynomial time algorithm for deciding the \emph{dead-beat} stability of a system,
 i.e. that all trajectories vanish to the origin in finite time.
  The algorithm generalizes one proposed by Gurvits for arbitrary switching systems, and is illustrated with a real-world case study.
\end{abstract}

\end{frontmatter}

\section{Introduction}
\label{sec:intro_ref}

Switching systems are dynamical systems for which the state 
dynamics themselves vary between different operating modes according to a \textit{switching sequence}. 
The systems under study in this paper are \textit{discrete-time linear switching systems}. 
Given a set of $N$ matrices $\mset = \{A_1, \ldots, A_N \} \subset \reels^{n \times n}$, 
the dynamics of a discrete-time linear switching system are given by $$ x_{t+1} = A_{\sigma(t)} x_t,$$
where $x_0 \in \reels^n$ is a given initial condition. 
The mode of the system at time $t$ is  $\sigma(t) \in \{1, \ldots, N\}$.
 The \emph{switching sequence}  driving the system is the sequence of modes $
\sigma(0), \sigma(1), \ldots$ in time. 
Such systems are found in many practical and theoretical domains. For example they appear in the study of networked control systems \cite{AlRaCMAA, JuDIFSOD}, in  
congestion control for computer networks \cite{ShWiAPSM},  
in viral mitigation \cite{HeMiOAMS}, 
as abstractions of more complex hybrid systems  \cite{LiMoBPIS}, and other fields (see e.g. \cite{JuTJSR, LiAnSASO, ShWiSCFS}  and references therein). 

 A large research effort has been devoted to the study of the  stability and stabilization  of switching systems 
 (see e.g.~\cite{LeDuUSOD, EsLeCOLS, LiAnSASO, ShWiSCFS, KuSSSF, DaBePDLF}). 
  The question of deciding the stability of a switching system is challenging and is known to be hard in general (see \cite{JuTJSR}, Section 2.2, for hardness results).\\
 In this paper, we first develop a sufficient condition for the boundedness of switching systems, that is,
  the existence of a uniform bound $K \geq 1$ such that for all switching sequences and all time $t \geq 0$,
\begin{equation}
   \|A_{\sigma(t)} \cdots A_{\sigma(0)}\| \leq K.
   \label{eq:boundedness}
\end{equation}  
 Second,  we provide an algorithm for deciding when a system is \emph{dead-beat} stable. 
 This corresponds to the fact that there exists a time $T \geq 1$ such that, for all switching sequences, and all $t \geq T$, $A_{\sigma(t)} \cdots A_{\sigma(0) }= 0$.
  Both problems have been studied for \emph{arbitrary switching systems} (\cite{BeWaBSOM, JuTJSR, BlTsTBOA, GuSOLIP2}), 
 for which the mode $\sigma(t)$ can take any value in $\{1, \ldots, N\}$ at any time. 
 To the best of our knowledge, these studies have yet to be extended to switching systems with more general switching sequences, 
 such as the ones studied in \cite{AhJuPaRoJSRP, EsLeCOLS,HeMiOAMS,KoTBWF, KuSSSF, LeDuUSOD}.

In this work, we allow for the definition of constraints on the switching sequences. These rules on the switching sequences are expressed through an \textit{automaton}.
 Automata are  common tools for the representation of admissible sequences of symbols (see \cite{LoACOW}, Section 1.3 for an introduction).
An automaton is here represented as a \emph{strongly connected} graph $\graph(V,E)$, with a set  of nodes $V$ and edges $E$. 
The edges of this graph are both \textit{directed} and \textit{labeled}. 
 An edge takes the form $(v, w, \sigma) \in E$, where $v$ and $w$ are respectively the \emph{source} and \emph{destination} nodes of the edge, 
and $\sigma \in \{1, \ldots, N\}$ is the \emph{label}, taking its values in the set of modes of the switching system. \\
The edges $E$ represent the possible time transitions of a switching system, and the nodes $V$ need not be inherently associated with modes of the system.
A switching sequence $\sigma(1), \sigma(2), \ldots,$ of the system is then said to be  \textit{accepted} by $\graph$ if there exists a path in $\graph$ 
such that the sequence of labels \emph{along the edges} of the path equals the switching sequence itself.
 Examples of such automatons, with their corresponding switching rules, are presented in Figure \ref{fig:example}.
 Note that, in general, there can be several automata that represent a same set of switching rules.  
 Because it is not relevant to our purpose, we do not specify an initial or final node for the paths - in this we differ  from the classical definition of an automaton
 (\cite{LoACOW}, Section 1.3). 
 \begin{figure}[ht]
\centering
\begin{subfigure}[t]{0.25\textwidth}
\centering
\includegraphics[scale= 0.4]{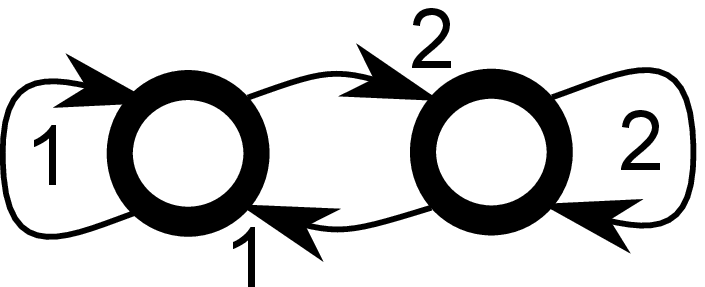}
\caption{}
\label{subfig:bigarb}
\end{subfigure}~
\begin{subfigure}[t]{0.25\textwidth}
\centering
\includegraphics[scale= 0.4]{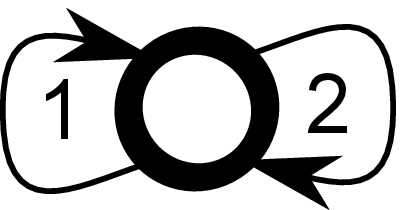}
\caption{}
\label{subfig:smallarb}
\end{subfigure}~
\begin{subfigure}[t]{0.25\textwidth}
\centering
\includegraphics[scale= 0.4]{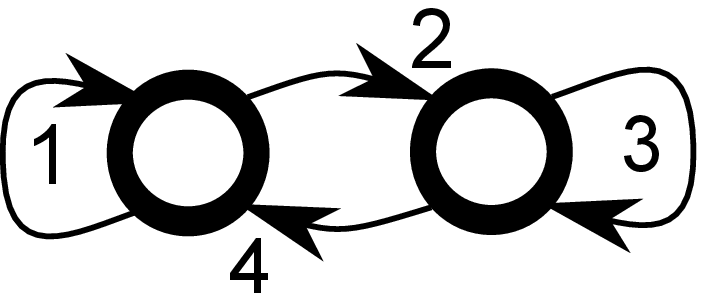}
\caption{}
\label{subfig:last}
\end{subfigure}

\caption{Both automata on Figure \ref{subfig:bigarb} and \ref{subfig:smallarb} accept \emph{arbitrary switching sequences} on two modes.
The third automaton (Figure \ref{subfig:last}) is on 4 modes. It does not generate arbitrary switching sequences, mode 1 must be followed by mode 1 or 2, mode 2 by mode 3 or 4, etc...}
\label{fig:example}
\end{figure}

Given a graph $\graph(V,E)$ on $N$ labels and a set of $N$ matrices $\mset$, we define the \emph{constrained switching system} $S(\graph, \mset)$
as the following discrete-time linear switching system:
\begin{equation}
\begin{aligned}
& x_{t+1} = A_{\sigma(t)} x_t, \, \\   & \qquad  \sigma(0), \ldots, \sigma(t) \text{ is accepted by $\graph$}.
\end{aligned}
\label{eq:dynsys}
\end{equation}
\emph{Arbitrary switching systems} are special cases of constrained switching systems.
 Their switching rule can be represented by the automata alike those of Figures \ref{subfig:bigarb} and \ref{subfig:smallarb}, 
 or by any \emph{path-complete graph} (see \cite{AhJuPaRoJSRP}).

The boundedness and stability properties of a system $S(\graph, \mset)$ are tightly linked to its \emph{constrained joint spectral radius}.
This concept was introduced by X. Dai \cite{DaAGTS} in 2012 for the stability analysis of constrained switching systems.
The CJSR of the system $S(\graph, \mset) $ is defined as follows:
\begin{equation}
 \begin{aligned}
 & \jsr(S) =
 \lim_{t \rightarrow \infty} \max \Big \{ \left \| A_{\sigma(t-1)}^{} \cdots A_{\sigma(0)}^{}
 \right \|^{1/t} : \, \\ 
  & \qquad \qquad
 \left .  \sigma(0), \ldots, \sigma(t-1)  \text{ is accepted by $\graph$} \right . \Big \}.
\end{aligned}
 \label{eq:CJSR}
\end{equation}  
 When $\graph$ allows for \emph{arbitrary switching}, the CJSR is equal to the 
 \emph{joint spectral radius} (JSR) of the set $\mset$, 
 which was introduced by Rota and Strang in 1960 (see \cite{JuTJSR} for a monograph on the topic). \\
The constrained joint spectral radius is the maximal exponential growth rate of a system.
 Its value reflects the stability properties of a system $S(\graph, \mset)$. 
If $ \jsr(S) < 1$, the system is both asymptotically and exponentially stable (see \cite{DaAGTS} - Corollary 2.8).
If $\jsr(S) > 1$, the system possesses an unbounded trajectory whose growth rate is exponential.
The last case $\jsr(S) = 1$ is more complicated.
The system is not asymptotically stable, but may be bounded or not depending on its parameters. \\
If $S$ is an arbitrary switching system with $\jsr(S) = 1$, then there exists a condition guaranteeing its boundedness.
This condition is the \emph{irreducibility} of $\mset$. 
\begin{definition}
A set $\mset \subset \reels^{n \times n}$ of matrices is \emph{irreducible} if
 for any non-trivial linear subspace $\subs{X} \subset \reels^n$, 
 i.e. with $0 < \dim(\subs{X}) < n,$
  there is a matrix $A \in \mset$ such that 
$ A\subs{X}\not \subset \subs{X}.$
That is, the matrices in $\mset$ do not share a common non-trivial invariant subspace of $\reels^n$.
\label{def:red}
\end{definition}
\begin{proposition}[e.g., \cite{JuTJSR}, Theorem 2.1]
If a set of matrices $\mset$ with joint spectral radius equal to 1 is irreducible,  then there is a uniform bound on the norm of all products of matrices taken from $\mset$.
\label{prop:arbbound}
\end{proposition}
The irreducibility property is known to be decidable (see \cite{ArPeTCIS} and references therein).
Moreover, irreducibility implies the existence of \emph{Barabanov} norms for arbitrary switching systems. 
The existence of such norms is very useful for the stability analysis of arbitrary switching systems \cite{WiOTSO, JuTJSR, GuPrECOJ}.\\
 The first part of this paper is focused around providing a proper generalization of both irreducibility and Proposition \ref{prop:arbbound} for
  constrained switching systems\footnote{Preliminary results on this question were presented in the conference paper \cite{PhJuASCF}.}.
As shown in Example \ref{examp:irrednotnondef}, Proposition \ref{prop:arbbound} does not generalize directly to constrained switching systems.

\begin{example}
\label{examp:irrednotnondef}
We construct an \emph{unbounded} constrained switching system $S(\graph, \mset)$, 
which has both $\jsr(S) = 1$ and an irreducible set of matrices.
The set $\mset = \{A_1, A_2\}$ is defined as
 \begin{equation}
 \begin{aligned}
 & A_1 = 
 \begin{pmatrix}
 1 & 1 \\
 0 & 1
 \end{pmatrix},
 & A_2 = 
 \begin{pmatrix}
 0 & 1 \\
 1 & 0
 \end{pmatrix}.
 \end{aligned}
 \end{equation}
 The automaton $\graph$ is given in Figure \ref{fig:auto_ex1}.
  It refuses any switching sequence containing $\ldots 121 \ldots$ as a subsequence.
 \begin{figure}[!ht]
 \centering
 \includegraphics[scale=0.4]{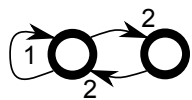}
 \caption{No paths in this graph can carry a sequence containing $\ldots 121 \ldots$.}
 \label{fig:auto_ex1}
 \end{figure}
Since $A_1$ and $A_2$  do not share common eigenvectors, the set $\mset$ is irreducible.\\
Moreover, we can verify that $\jsr(S) = 1$. Indeed, for all $k \geq 1$, 
the sequence $\sigma(1), \ldots, \sigma(k)$ maximizing the norm 
$\|A_{\sigma(k)} \cdots A_{\sigma(1)}\|$ in (\ref{eq:CJSR}) is simply $\sigma(t) = 1$ for all $t = 1, \ldots, k$. 
Equation (\ref{eq:CJSR}) thus boils down to the classical Gel'fand formula for computing the spectral radius 
of $A_1$, which is equal to \emph{1}.\\
 However, for the same switching sequence, we have that $\|A_1^t\| \rightarrow \infty$ as $t \rightarrow \infty$, and so $S$ is unbounded.\\
We conclude that proposition \ref{prop:arbbound} does not generalize as it is to constrained switching systems.
\end{example}

In the second part of the paper, we put ourselves in the case $\jsr(S) = 0$. From the results of Dai, we known that this implies the asymptotic stability of any constrained switching system. 
However, in the arbitrary switching case, this is also equivalent to \emph{dead-beat} stability (Theorem \ref{thm:0JSR}) and we can decide $\jsr(S) = 0$ in polynomial time. 
\begin{theorem}[e.g.~\cite{JuTJSR}, Section 2.3.1]
The joint spectral radius of a set of $N$ matrices $\mset$ 
is zero if and only if all
 products of size $n$ of matrices in $\mset$ are equal to zero.
\label{thm:0JSR}
\end{theorem} 
To the best of our knowledge, 
no extensions of these results to constrained switching systems have been proposed yet.
 And again, it appears that Theorem \ref{thm:0JSR} does not hold for constrained switching systems.
 \begin{example}
The \emph{scalar} ($n = 1$) system on the two matrices $A_1 = 1$ and $A_2 = 0$ with the cyclic automaton of Figure \ref{fig:cycle} 
sees all of its trajectories vanish to the origin after 2 steps, and so $\jsr(S) = 0$ from (\ref{eq:CJSR}).
 \begin{figure}[!ht]
 \centering
 \includegraphics[scale=0.4]{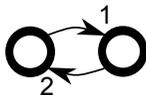}
 \caption{A simple cyclic system on two modes.}
 \label{fig:cycle}
 \end{figure}
We conclude that Theorem \ref{thm:0JSR} does not hold when constraints are added on switching sequences.
\end{example}

The plan of the paper is as follows.\\
Section \ref{sec:bound} presents a decidable sufficient condition for the boundedness of systems $S(\graph, \mset)$ with $\jsr(S) = 1$.
Our main result provides a condition generalizing  the concept of irreducibility to constrained switching systems. 
 Its decidability is proven in Subsection \ref{sub:dec}.  
In Subsection \ref{sec:koz}, we then compare the condition with a second one which, after a one-shot transformation 
of a constrained switching system into a higher-dimensional arbitrary switching system, relies on the classical definition of irreducibility .\\
In Section \ref{sec:0cjsr}
 we generalize Theorem \ref{thm:0JSR} to constrained switching systems. We provide an algorithm that, 
 given a system $S(\graph,  \mset)$, outputs whether $\jsr(S) = 0$ or $\jsr(S) > 0$.
  The complexity of the algorithm is polynomial in the dimensions of $\mset$ and $\graph$. 
  An example of application of this algorithm for designing the left-inverter of a dynamical system is presented in Subsection \ref{ssec_electric}.

\subsection*{Notations}
Given an automaton $\graph(V,E)$ with nodes $V$ and edges $E$, $\card{V}$ and $\card{E}$ denote respectively the amount of nodes and edges of $\graph$.
A path in the graph $\graph$ is a sequence of \emph{consecutive} edges.
The length of a path $p$ is denoted  $|p|$ and is the amount of \emph{edges} it contains. An edge $e = (v_1, v_2, \sigma)$ is  a path of length 1.
We refer to \emph{sets of paths} by borrowing notations of language theory.
 The set of all the paths of length $k$ is written $E^k$, 
 and we let $E^* = \cup_{k = 1}^{\infty} E^k$. 
 Given two nodes $v$ and $w$, we use $E^k_{v,w}$ to denote the set of all the paths of length $k$ from the node $v$ to the node $w$,
  and we let $E_{v,w}^* = \cup_{k = 1}^{\infty} E_{v,w}^k.$
  When the second node is unspecified, $E_{v}^k$ and $E_{v}^*$ relate to paths whose source is $v$, and whose destination can be any  node in $V$. 
We use $p = [p_1, p_2, \ldots]$ to denote a partition of the path $p$ in smaller consecutive paths $p_1$, $p_2$, $\ldots$.
A \emph{cycle} on a node $v$ is a path $c \in E^*_{v,v}$. The cycle $c$ is \emph{simple} on the node $v$ if there are no partitions $c = [c_1, c_2]$ where both $c_1 \in E_{v,v}^*$ and $c_2 \in E_{v,v}^*$.\\
To each path $p$ with length $\card{p} = T$ is associated a matrix product $A_p = A_{\sigma(T)} \cdots A_{\sigma(1)}, $
where $\sigma(i)$, $1 \leq i \leq t$ is the label of the $i$'th edge in $p$.\\
For $x \in \reels^n$ and $A \in \reels^{n \times n}$, $\| x \|$ is the euclidean norm of $x$ and $\|A\|$ the matrix norm of $A$ induced by the euclidean norm.

\section{Generalizing irreducibility to constrained switching systems}
\label{sec:bound}

In this section, we provide a sufficient condition for the boundedness of systems with $\jsr(S) = 1$,  
generalizing the concept of irreducibility and Proposition \ref{prop:arbbound} to constrained switching systems. 
The first part of this section is focused on the presentation our main result Theorem \ref{thm:non_def2}. 
Its proof is detailed in Subsection \ref{subsec:proof}. In Subsection \ref{sub:dec} we prove the decidability of the condition. 
Subsection \ref{sec:koz} then presents an alternative boundedness condition, 
appealing since it allows to ignore combinatorial details about the graph $\graph$ after a one-shot transformation of a system, 
but definitely weaker than Theorem \ref{thm:non_def2}.

Our generalization of irreducibility relies on the following definition.
\begin{definition}[Nodal irreducibility]
 Given a system $S(\graph(V,E), \mset)$,
  a \emph{node} $v \in V$ is \emph{irreducible} if for any \emph{non-trivial} linear subspace $\subs{X} \subset \reels^n$ 
 there is a cycle $c \in E^*_{v,v}$ such that
$ A_c \subs{X} \not \subset \subs{X}. $
\label{def:ni}
\end{definition}

For a system with $\jsr(S) = 1$, it must be the case that there
 exists a sequence of paths in $\graph$, $\{p_k \in E_{v}^k\}_{k = 1,2,\ldots}$, starting at some node $v \in V$ such that,
  for all $k = 1,2,\ldots$, $A_{p_k} \neq 0$. 
The existence of arbitrary long paths is guaranteed by the strong connectivity of the directed graph $\graph$. 
But this is not enough.  The system $S(\graph, \mset)$ must possess a connectivity beyond the one brought in a graph theoretical sense by $\graph$. 
\begin{definition}[Linear Connectivity]
Given a system $S(\graph(V,E), \mset)$ and two nodes $v, w \in V$, 
$v$ is \emph{linearly connected} to $w$  if there exists $p \in E^*_{v,w}$ such that $A_p \neq 0.$\\
The \emph{system} $S$ is \emph{linearly connected} if all pairs of nodes are linearly connected. 
\label{def:lc}
\end{definition}

Our main result in this section can be roughly stated as follows. 
If we can guarantee that all the trajectories of a system with CJSR = 1 encounter 
irreducible nodes regularly, then the system is bounded.
 To formalize the result, we borrow the concept of unavoidability from automata theory (see e.g. \cite{LoACOW}, Proposition~1.6.7.).
\begin{definition}
Given a graph $\graph(V,E)$, a subset $V^*$ of $V$ is said to be unavoidable if any cycle in $\graph$ encounters at least one node in $V^*$.
\end{definition}
\begin{lemma}
Given a graph $\graph(V,E)$, let $V^*$ be an unavoidable set of nodes.
Any path $p$ of length $\card{V}$ encounters at least 1 node in $V^*$.
\label{lemma:unavoidable}
\end{lemma}
\begin{proof}
Such a path $p$ must visit $\card{V}+1$ nodes. 
At least one node in $V$ appears twice in $p$, implying that $p$ contains a cycle. Therefore, $p$ visits an unavoidable node in $V^*$.
\end{proof}
Our main result is the following.
\begin{theorem}
Consider a  system $S(\graph(V,E), \mset)$  that is linearly connected and has an unavoidable set of irreducible nodes. 
If $\jsr(S) = 1$, then the system is bounded.
\label{thm:non_def2}
\end{theorem}
The conditions guaranteeing boundedness in Theorem \ref{thm:non_def2} are split in two sets of conditions.\\
The first set of conditions (irreducibility and connectivity) is on the \emph{structure}  of the system $S$. 
As we will prove in Subsection \ref{sub:dec}, these conditions are decidable.\\
The second condition is $\jsr(S) = 1$.
Given any system $S$, the question of knowing whether $\jsr(S) = 1$ or not is known to \emph{undecidable} (see \cite{BlTsTBOA} for the special case of arbitrary switching). 
However if we relax this condition, we may still extract some interesting properties about the stability of $S$.
 This is discussed in Section 3, Proposition \ref{prop:non-def}.

In the previous conference version of this work \cite{PhJuASCF}, the concept of \emph{unavoidability} was left aside.
The main result of \cite{PhJuASCF} is in fact a special case of Theorem \ref{thm:non_def2}, where all the nodes of $\graph$ are irreducible.
This trivially implies the unavoidability of the set of irreducible nodes.

\begin{theorem}[ From \cite{PhJuASCF}, Theorem 2.1]
Consider a  system $S(\graph(V,E), \mset)$  that is linearly connected and has all of its nodes being irreducible. 
If $\jsr(S) = 1$, then the system is bounded.
\label{thm:non-def}
\end{theorem}

There are scenarios where the irreducibility of all nodes (Theorem \ref{thm:non-def}) is definitely too strong of a requirement, 
  but where systems do possess an unavoidable set of irreducible nodes (Theorem \ref{thm:non_def2}).

\begin{example}
Consider a system on two matrices $\mset=\{ A_1, A_2\}$
with 
$$A_1 = \begin{pmatrix} 0 & 1 \\ 1 & 0\end{pmatrix}, \, A_2 = \begin{pmatrix}
0 & 0 \\
0 & 1
\end{pmatrix}, $$
and constrained by the automaton of Figure \ref{fig:ex_failing}. 
\begin{figure}[!ht]
\centering
\includegraphics[scale = 0.4]{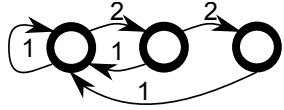}
\caption{We cannot activate mode 2 more than twice in a row.}
\label{fig:ex_failing}
\end{figure}
The node on the \emph{left} is irreducible since the matrices $A_1$ and $A_1A_2$ do not share common eigenvectors. The node is also \emph{unavoidable} since contained in every cycle in the graph.
However, the middle and right nodes are \emph{reducible}, with the invariant subspace $\im{A_2}$.
In conclusion, the system  satisfies to the conditions of Theorem \ref{thm:non_def2} and has only one irreducible node.
 \label{example:weakness}
\end{example}

We now move to the proof of Theorem \ref{thm:non_def2}.

\subsection{The proof of Theorem \ref{thm:non_def2}}
\label{subsec:proof}

We start with two lemmas.
\begin{lemma}
Consider a system $S(\graph, \mset)$, linearly connected and has an unavoidable set of irreducible nodes $V^* \subset V$. At any irreducible node $v \in V^*$, the subspace defined as 
\begin{equation}
\subs{K}_v = \{x \in \reels^n: \exists K < \infty: \forall p \in E_{v}^*, \|A_p x\| < K \} 
\label{eq:boundedSubspace}
\end{equation}
is either $\reels^n$ or $\{0\}$. Moreover, if $S$ is unbounded, then for any irreducible node $v \in V^*$, $\subs{K}_v = \{0\}$.
\label{lemma:ubndqualify}
\end{lemma}
\begin{proof}
It is clear that $\subs{K}_v$ is a linear subspace. We prove the first part by contradiction. Take an irreducible node $v \in V^*$, and assume first that
$\subs{K}_v \neq \reels^n$, and second that there is a non-zero vector $ x \in \subs{K}_v$.
Since $v$ is irreducible, there must be $c \in E_{v,v}^*$ such that 
$y = A_c x \not \in \subs{K}_{v}$. Indeed, if there was no such cycle, the subspace defined as
$$\spans\{ x, \cup_{c \in E_{v,v}^*} A_c x \} \subseteq \subs{K}_{v} $$
would be invariant at $v$, contradicting the irreducibility of the node.
This implies the existence of unbounded trajectories from the point $y$, which in turn implies the contradiction $x \not \in \subs{K}_v$. \\
We now move to the second part of the Lemma and assume that $S$ is unbounded. In that case,
 there exists a node $v$ (not necessarily irreducible) and a sequence of paths $ \{p_k \in E_{v}^k\}_{k = 1,2,\ldots}$ 
 such that $\lim_{k\rightarrow \infty} \|  A_{p_k} \| = \infty$.\\ 
For $k \geq |V|$, $p_k$ must have visited an \emph{irreducible} node $w \in V^*$ (Lemma \ref{lemma:unavoidable}).
  Therefore, we have $\subs{K}_w = \{0\}$ from the first part of the Lemma.\\
If $w$ was the only irreducible node, then the proof is done. So assume there are other irreducible nodes, and take such a node $w'  \in V^*$.
 From the linear connectivity of $S$, there must be a path $p \in E_{w',w}^*$ and a vector $x \in \reels^n$ such that $A_p x  \neq 0$, 
 and since $\subs{K}_w = \{0\}$, $A_p x \not \in \subs{K}_w$. 
 Therefore it must be the case that $x \not \in \subs{K}_{w'}$ and so $\subs{K}_{w'} \neq \reels^n$.
  From the first part of the Lemma, we must conclude $\subs{K}_{w'} = \{0\}$. 
  Repeating the above for all irreducible nodes, we then conclude that $K_v = \{0\}$ for all $v \in V^*$.
\end{proof}

\begin{lemma}
Consider a system $S(\graph, \mset)$, linearly connected and with an unavoidable set of irreducible nodes $V^* \subset V$.
 If $S$ is unbounded then  for any scalar $\ell > 1$, 
 there must exist an integer $T_\ell \geq 1$ such that for any irreducible node $v \in V^*$ and  any  $x \in \reels^n,$
\begin{equation}
\exists\, p \in E_{v_T}^{*}, \card{p} \leq T_\ell: \|A_p x\| \geq \ell \|x\|.
\label{eq:proof_fromirredtoany}
\end{equation}
\label{lemma:ubndquantify}
\end{lemma}
\begin{proof}

Assume by contradiction that the opposite of the Lemma's claim is true. 
There exists a scalar $\ell  > 1$ such that for all $T \geq 1$, there is an irreducible node $v_T \in V^*$ and a vector $x_T \neq 0$
satisfying
\begin{equation*}
\forall\, p \in E_{v_T}^{*}, \card{p} \leq T: \|A_p x_T\| < \ell \|x_T\|.
\end{equation*}
We can further assume, without loss of generality, that for all $T$, $\|x_T\| = 1$.
Since there are a finite amount of nodes in $V^*$, there must be a sequence $\{T_k\}_{k = 1, 2, \ldots}$ and a node $w \in V^*$, such that $v_{T_k} = w$ for all $k = 1, 2, \ldots$.
The vectors $\{x_{T_k}\}_{k = 1, 2, \ldots}$ form a  \emph{sequence} of vectors  on the unit ball $\{x \in \reels^n:\, \|x\| = 1\}$.
 Since the unit ball is a \emph{compact} set, there exists a subsequence of  $\{x_{T_k}\}_{k = 1, 2, \ldots}$ converging,
  as $k$ goes to infinity, to a limit point $x_*$ on the unit ball. This limit point is such that
$$\forall p \in E_{w}^*: \, \|A_p x_* \| \leq \ell \|x_*\|,$$
which implies by definition that $x_* \in \subs{K}_w$. 
By Lemma \ref{lemma:ubndqualify}, this contradicts the fact that, since $w$ is irreducible, it should have $\subs{K}_v = \{0\}$.

\end{proof}

\begin{proof}[Proof of Theorem \ref{thm:non_def2}]
We consider a system $S(\graph(V,E), \mset)$ which is linearly connected and with an unavoidable set of irreducible nodes  $V^* \subset V$. 
We will show that if $S$ is \emph{unbounded}, then $\jsr(S) > 1$. Theorem  \ref{thm:non_def2} will then be proven by contraposition.

To do so, assume that $S$ is \emph{unbounded}.
We begin by proving a result similar to Lemma \ref{lemma:ubndquantify}, 
but which  only involves irreducible nodes.

\emph{Claim: There exists an integer $T^* \geq 1$ such that, for any irreducible node $v \in V^*$, and any non-zero vector $x \in \reels^n,$}
\begin{equation}
\exists w \in V^*, \, \exists\, p \in E_{v,w}^{*},\,  \card{p} \leq T^*:\, \|A_p x\| \geq 2 \|x\|.
\label{eq:proof_fromirredtoirred}
\end{equation}

For proving the Claim, we first define $K  = \max_{A  \in \mset} \|A\|$.
Since  induced matrix norms are sub multiplicative, 
we have that for all $p \in E^*$, $\|A_p\| \leq K^{\card{p}}$.
Also $K > 1$ since we assumed $S$ to be unbounded.\\
Observe that Lemma \ref{lemma:ubndquantify} holds under our current set of assumptions. 
Moreover, for any choice of the constant $\ell$, paths $p$ satisfying (\ref{eq:proof_fromirredtoany}) must be such that $$ \ell \leq \| A_p\| \leq K^{\card{p}}.$$
Therefore, for the choice $\ell = 2 K^{\card{V}},$
 we can guarantee that the paths in (\ref{eq:proof_fromirredtoany}) have a \emph{minimum length} of $\card{V} + 1$. \\
We now invoke the concept of unavoidability.
 Lemma \ref{lemma:unavoidable} guarantees that an irreducible node is encountered within the last $\card{V}$ edges of the paths considered above.
Take any irreducible node $v$ and any $x \in \reels^n$. Let $p$ be a path satisfying (\ref{eq:proof_fromirredtoany}) for $\ell = 2K^{\card{V}}$. 
Let $w \in V^*$ be the last irreducible node encountered along the path, and let $v' \in V$ be the destination of the path. Assume for the moment that $v' \neq w$.
 We can write $p = [p_1, p_2]$, with $p_1 \in E_{v,w}$ and $p_2 \in E_{w,v'}$, and with $\card{p_2} \leq \card{V}$. 
By definition of $p$, from our choice of $\ell$, and by sub multiplicativity of the norm, we have 
$$ 2K^{\card{V}} \|x\| \leq \|A_{p_2}A_{p_1}x\|\leq K^{\card{p_2}} \|A_{p_1}x\| \leq  K^{\card{V}} \|A_{p_1}x\|. $$
If $w = v'$, the discussion is even simpler and we let  $p_1 = p$, $\|A_{p_1}x\| \geq  2K^{\card{V}} \|x\|$.
We have shown the existence of a path $p_1$ whose origin is $v$ and destination is an irreducible node $w$ such that $\|A_{p_1}x\|\geq 2 \|x\|$.
This being done from an arbitrary irreducible node $v$ and arbitrary $x \in \reels^n$, we conclude that our Claim holds true.

We now prove that $\jsr(S) > 1$.
Take any irreducible node $v_0 \in V^*$ and any non-zero $x \in \reels^n$.
From (\ref{eq:proof_fromirredtoirred}),  there exists a sequence $\{v_k\}_{k = 1,2,\ldots}$ of nodes in $V^*$ and a corresponding sequence of paths $\{p_k\}_{k = 1,2, \ldots}$, each with a length $T_k \leq T^*$ and with $p_k \in E^{T_k}_{v_{k-1}, v_k}$, that satisfy for all $k = 1,2,\ldots$, 
$$ \|A_{p_k} \cdots A_{p_1} x\| \geq 2^{k} \|x\|.$$
We then obtain a lower bound on $\jsr(S)$ from (\ref{eq:CJSR}):
$$ \jsr(S) \geq \lim_{ k \rightarrow \infty } \left ( \|A_{p_k} \cdots A_{p_1} \|^{1 / \sum{T_k} } \right ) \geq 
\lim_{ k \rightarrow \infty } \left (  \|A_{p_k} \cdots A_{p_1} \|^{ 1/{(k T^*)} } \right ) \geq 2^{1 /T^*} > 1.
 $$  
 We conclude that a system $S(\graph, \mset)$, linearly connected and with an unavoidable set of nodes is \emph{unbounded} if $\jsr(S) > 1$. If $\jsr(S) \leq 1$, the system must be bounded.
\end{proof}

\subsection{Decidability of nodal irreducibility and linear connectivity}
\label{sub:dec}

Knowing whether a finite set of matrices $\mset$ is irreducible is decidable 
(see e.g. \cite{ArPeTCIS} and references therein). That is, given a \emph{finite} set of matrices $\mset$, there exists an algorithm, running in finite time,  returning whether the set is irreducible or not.\\
In this subsection, we show that knowing whether  a system $S(\graph, \mset)$ is linearly connected and has an unavoidable set of irreducible nodes is decidable as well. To do so, we show that the irreducibility of a single node in the system is equivalent to the irreducibility of a finite set of matrices, and that a similar result holds for linear connectivity. 

We begin by a technical lemma. The authors believe the result to be folklore in linear algebra, but its proof, while perhaps not presented in its most compact form, reveals to be insightful for the proof of Theorem~\ref{thm:bound}.

\begin{lemma}
Let $\mset = \{A_1, A_2, \ldots \}$ be a (possibly infinite) set of matrices in $\reels^{n \times n}$. 
Given a non-trivial subspace $\subs{X} \subset \reels^n$ of dimension $1 \leq d < n$ and $x \in \subs{X}$, $x \neq 0$, 
if for all products of matrices of length $t \leq d$,
$$ A_{\sigma(t)} \cdots A_{\sigma(1)} x \in \subs{X} ,$$
with $A_{\sigma(i)} \in  \mset$, $1 \leq i \leq t$ then for any product $A^*$ from matrices in $\mset$,
 $$ A^* x \in \subs{X}.$$
 \label{lemma:exitSloop}
\end{lemma}

\begin{proof}
We show by contradiction that the shortest product under which the image of $x$ no longer lies in $\subs{X}$ must be of length at most $d$.
Assume that $T > d$ is the length of a shortest product such that
\begin{equation}
 A^* A_{\sigma(d)} \cdots A_{\sigma(1)} x \not \in \subs{X},
 \label{eq:lemmasubspaceLeave}
\end{equation} 
where $A^*$ is itself a product of length $T-d$ of matrices in $\mset$. \\
Since the vectors $x, A_{\sigma(1)}x, \ldots, A_{\sigma(d)}\cdots A_{\sigma(1)}x$ are $d+1$ vectors that belong to $\subs{X}$ which is of dimension $d$, we can state the following:
\begin{equation}
 \exists\, 1 \leq k \leq d : \,  A_{\sigma(k)}  \cdots A_{\sigma(1)} x \in \spans \left \{ x, A_{\sigma(1)} x,\ldots,  A_{\sigma(k-1)}\cdots A_{\sigma(1)} x \right \}.
 \label{eq:lemmasubspaces}
\end{equation}
Assume for the moment that $k = d$. From (\ref{eq:lemmasubspaceLeave}) and (\ref{eq:lemmasubspaces}),  we can write
$$A^* A_{\sigma(d)} \cdots A_{\sigma(1)} x \in  
\spans \left \{ A^* x, A^* A_{\sigma(1)} x,\ldots,  A^* A_{\sigma(d-1)}\cdots A_{\sigma(1)} x \right \} \not \subseteq \subs{X}. $$
After observing that all products of matrices involved in the span on the right are of length at most $T - 1$,
 and that at least one vector in this span does not lie in $\subs{X}$, we conclude that $T$ is not the shortest length.\\
In the case $k < d$, we can repeat the above process by first multiplying (\ref{eq:lemmasubspaces}) on the left by $\prod_{i = k+1}^d A_{\sigma(i)}$ and then by $A^*$. The conclusion remains unchanged, and we have proved the Lemma by contradiction.
\end{proof}

\begin{theorem}
 Consider a system $S(\graph(V,E), \mset)$, $\mset \subset \reels^{n \times n}$.
 Given an irreducible node $v$ and a non-trivial linear subspace 
 $\subs{X} \subset \reels^n,$ the \emph{shortest cycle} $c \in E_{v,v}^*$ such 
 that $A_c\subs{X} \nsubseteq \subs{X}$ is of length 
 $$	|c| \leq 1+n(\card{V} - 1). $$  
 If a pair of nodes $v,w$ is linearly connected, then the shortest path $p \in E_{v,w}^*$ with $A_p \neq 0$ is of length
 $$ |p| \leq 1 + n(\card{V} - 2).$$
\label{thm:bound}
\end{theorem}
\begin{proof}
We start with the first part of the theorem. 
Consider an irreducible node $v$, a non-trivial subspace $\subs{X}$ and a cycle $c_v \in E_{v,v}^*$ \emph{with the smallest length} satisfying 
\begin{equation}
A_{c_v} \subs{X} \not \subseteq \subs{X}.
\label{eq:cvoutofX}
\end{equation}
We are now going to establish two properties of $c_v$, allowing us to show that $\card{c_v} \leq 1+n(\card{V}-1).$

\textit{Claim 1 : The cycle $c_v$ is \emph{simple} on the node $v$.}\\
The simplicity of $c_v$ means that there are no partitions $c_v = [c_v^1, c_v^2]$ where $c_v^1, \, c_v^2 \in E^*_{v,v}$. If there was such a partition, then  either $A_{c_v^1}\subs{X} \not \subset \subs{X}$, or $A_{c_v^2} \subs{X} \not \subset \subs{X}$. Thus, $c_v$ would not be a shortest cycle satisfying (\ref{eq:cvoutofX}).

\textit{Claim 2 : The  cycle $c_v$ never visits a node $w \neq v$ more than $n$ times.}\\
This part of the proof relies on Lemma \ref{lemma:exitSloop} and its proof.\\
Assume by contradiction that $c_v$ visits a node $w \neq v$ for $T > n$ times. We can write $$ c_v =[p_{v,w}, c_w, p_{w,v}],$$
where $p_{v,w} \in E^*_{v,w}$ and $p_{w,v} \in E^*_{w,v}$ are paths that visit $v$ and $w$ at most once as source or destination; and $c_w$ is a cycle on the node $w$. In turn, we can write $$ c_w = [c_w^1, \ldots, c_w^{T-1}],$$
where the $c_w^k$, $k = 1, \ldots, T-1$ are \emph{simple cycles} on $w$. 
Note that the above does imply that $w$ is visited $T$ times overall in $c_v$.\\
We now define the subspaces $\subs{X}_w = A_{ p_{v,w}} \subs{X}$ and  
$ \subs{Y} = \{ x \in \reels^n: A_{p_{w,v}} x \in \subs{X}\}. $ Since $c_v$ satisfies (\ref{eq:cvoutofX}) and is of shortest length, these subspaces must satisfy to the following two relations:
\begin{equation*}
\begin{aligned}
& \subs{X}_w \subseteq \subs{Y}, \\
& A_{c_w} \subs{X}_w \not \subseteq \subs{Y},
\end{aligned}
\end{equation*}
and consequently $ 1 \leq \dim(\subs{Y}) \leq n-1 $. The second relation above is equivalent to
$$ \left ( A_{c_w^{T-1}} \cdots A_{c_w^n} \right ) A_{c_w^{n-1}} \cdots A_{c_w^1} \subs{X}_w \not \subseteq \subs{Y},$$
and from the proof of Lemma \ref{lemma:exitSloop}, there must exist $1 \leq k \leq n-2$ such that
 $$ \left ( A_{c_w^{T-1}} \cdots A_{c_w^n} \right ) A_{c_w^{k}} \cdots A_{c_w^1} \subs{X}_w \not \subseteq \subs{Y}.$$
 This contradicts the fact that $c_v$ is of shortest length. Indeed, defining $c_w' = [c_w^1, \ldots, c_w^k, c_w^n, \ldots c_w^{T-1}]$, 
 we have that $c_v' = [p_{v,w}^{}, c_w', p_{w,v}^{}]$ satisfies (\ref{eq:cvoutofX}) and is such that $|c_v'| < |c_v^{}|$. This concludes the proof of the second claim.
 
Having established these properties, we can easily compute an upper bound on the length of $c_v$: there are at most $n(\card{V}-1)$ visits to nodes other than $v$,
 requiring a path of length at most $n(\card{V}-1)-1$. 
 To form a simple cycle on $v$ from this path,
  we add two edges from and to $v$, 
  respectively at the beginning and the end of the path, 
  giving us a cycle of length $1 + n(\card{V}-1)$ at most.

The second part of the theorem relates to linear connectivity. 
The proof is similar in its ideas, so we will only highlight the main differences.
Consider two linearly connected nodes $v$ and $v'$ and let the path $p_{v,v'} \in E^*_{v,v'}$ be a shortest path satisfying $A_{p_{v,v'}} \neq 0$. 
Similarly to Claim 1, we can show that $p_{v,v'}$ cannot visit neither $v$ or $v'$ more than once (as origin and destination, resp.).
Claim 2 still holds, any node $w$ visited by $p_{v,v'}$ cannot be visited more than $n$ times. 
The proof proceeds by writing $p_{v,v'} = [p_{v,w}, c_w, p_{w,v'}]$ as before,
defining $\subs{Y} = \ker{A_{p_{w,v'}}}$, 
and applying Lemma \ref{lemma:exitSloop} to escape  $\subs{Y}$ from $\im{A_{p_{v,w}}}$. 
Finally, $p_{v,v'}$ visits $\card{V}-2$ nodes $n$ times at maximum, 
to which we need to add two edges for connecting $v$ and $v'$,
 giving us a bound of $|p_{v,v'}| \leq 1+n(\card{V}-2)$.
\end{proof}

Based on the above result and on the fact that irreducibility (Definition \ref{def:red}) is decidable, 
we are able to formulate a decision algorithm for verifying that a system is linearly connected and has an unavoidable  set of irreducible nodes. 
We present this algorithm in three distinct parts. 
First (Algorithm \ref{algo:Connect}) checks for the linear connectivity of the system, 
second (Algorithm \ref{algo:Irred})  marks all the irreducible nodes,
and third (Algorithm \ref{algo:Unavoidability}) checks the unavoidability of the set of irreducible nodes.

\begin{algorithm}[!ht]
\caption{Checking linear connectivity of $S(\graph(V,E), \mset)$}
\begin{algorithmic}
\For {all pairs of nodes $v,w \in V$}
		\If {$\forall p \in E^*_{v,w}, \card{p} \leq 1+n(\card{V}-2): \, A_p = 0$}
		\State{Return False } \Comment{the pair v,w is disconnected.}
		\EndIf
\EndFor
\State{Return True } \Comment{all pairs are Linearly Connected.}
\end{algorithmic}
\label{algo:Connect}
\end{algorithm}
\begin{algorithm}[!ht]
\caption{Marking the irreducible nodes of $S(\graph(V,E), \mset)$}
\begin{algorithmic}
\State {Initialize set of irreducible nodes: $V^* = \{\}$.}
\For {all nodes $v \in V$}
\State{Initialize set of matrices $C_v = \{\}$.}
	\For{all cycles $c \in E_{v,v}^*$, with $\card{c} \leq 1+n(\card{V}-1)$}
	\State{Update the matrix set: $C_v \gets C_v \cup A_c$.}
	\EndFor
	\If{the set $C_v$ is irreducible}
	\State{Mark the node $v$ as irreducible: $V^* \gets V^* \cup v$.}
	\EndIf
\EndFor 
\State{Return $V^*$}
\end{algorithmic}
\label{algo:Irred}
\end{algorithm}
\begin{algorithm}[!ht]
\caption{Checking unavoidability of a subset $V^* \subset V$ in the graph $\graph(V,E)$}
\begin{algorithmic}
\State {Construct the graph $\graph'(V \backslash V^*, E')$ by removing all nodes in $V^*$ from $\graph$.}
\If {$\graph'$ has no cycle}
\State{Return True}
\Comment{the set $V^*$ is unavoidable}
\EndIf
\State{Return False}
\end{algorithmic}
\label{algo:Unavoidability}
\end{algorithm}

\begin{remark}
The decidability of nodal irreducibility is guaranteed from the fact that the irreducibility of matrix sets is decidable.
However, to the best of our knowledge, it is not known whether deciding the irreducibility of a set $\mset$ is NP-hard or not.
Moreover, even if we assumed the existence of an algorithm deciding the irreducibility of a set $\mset$ in a number of steps polynomial in the size of  $\mset$,  Algorithm \ref{algo:Irred} would still require an exponential amount of steps. Indeed, the algorithm proceeds by computing all cycles of length $1+n(\card{V}-1)$. 
Even for two nodes, the amount of cycles involved is $\bigO(\card{E}^{\left ( 1+n(\card{V}+1) \right )})$.
\label{rem:polytime}
\end{remark}
Before ending this subsection, we show that the bounds of Theorem \ref{thm:bound} are \emph{tight}.
\begin{proposition}
For all integers $n, m \geq 1$, one can build a system $S(\graph(V,E), \mset)$, $\mset \subset \reels^{n \times n}$,
 $\card{V} = m$,  
for which a cycle of length $1 + n(m-1)$ is needed to exit a particular non-trivial subspace $\subs{X} \subset \reels^n$ from a particular node $v \in V$. This holds even when $\mset$ contains two matrices.
\end{proposition}
\begin{proof}
We take any automaton $\graph$ on $m$ nodes from the \emph{{\v{C}}ern{\`y} automata} family 
(see Figure \ref{fig:cerny}). 
These automata are related to the well-known {\v{C}}ern{\`y}'s conjecture,
 which is still the subject of active research (a state-of-the-art survey is presented in \cite{VoSAAT}). 
 \begin{figure}[ht]
\centering
\includegraphics[scale= 0.4]{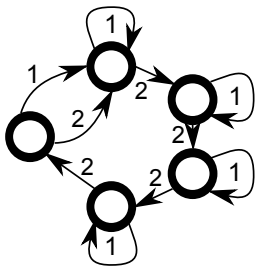}
\caption{A {\v{C}}ern{\`y} automaton on 5 nodes. In this example, the particular node $v$ is taken as the node on the left without self-loop. }
\label{fig:cerny}
\end{figure}
{\v{C}}ern{\`y}' automata possess $N = 2$ transition labels, and
we define following the set of matrices $ \mset = \{A_1, A_2\} \in \{0, 1\}^{n \times n}$:
\begin{align*}
& A_1(i,i+1) = 1, \,  i = 1 \cdots n - 1, \\
& A_2(n,1) = 1, 
\end{align*}
with the rest of the entries equals 0 for both matrices.  \\
 The  particular node $v \in V$ is taken to be the only node without self loop.
 The non-trivial subspace $\subs{X}$ is defined as the span of the basis vector $e_1 = [1,0,0, \ldots]^\top \in \reels^n$.
Notice that $e_1 \in \text{Ker}(A_1)$, $A_2 e_1 \in \text{Ker}(A_2)$, $A_2 e_1 \not \subset \subs{X}$, and finally $A_1^{n-1}A_2 e_1 = e_1$.\\
For this system, the shortest cycle $c \subset E$ on $v^*$ such that 
$y = A_c x \not \in \subs{X} $ has the associated product
$$A_c = A_2^{} \, (A_1^{n-1}A_2^{})^{m-1}_{}, $$
which is indeed a product of length $1 + n (m-1)$.

Regarding linear connectivity property, the bound $1 + n(m - 2)$ of Theorem \ref{thm:bound} is attained through the same family of automaton, with the same matrices, by taking the source node as the node without self loop, and the destination node as the one just before it (bottom of Figure \ref{fig:cerny}).
\end{proof}

\subsection{A sufficient condition based on an algebraic lifting}
\label{sec:koz}

In this last subsection, we show how a recently developed lifting technique 
can be put to good use in order to derive an alternative sufficient condition for the boundedness of a system.  We will also make use of the technique for the proofs presented in Section \ref{sec:0cjsr}.

\begin{definition}
Consider a system $S(\graph(V,E), \mset)$ with nodes $V = \{v_k\}_{k = 1, \ldots, \card{V}}$.
The set $\mset_\graph \subset \reels^{\left ( \card{V}n \right ) \times  \left ( \card{V}n \right )}$ is defined as follows :
$$\mset_\graph = \{A_{(v_i, v_j, \sigma)}, \, (v_i, v_j, \sigma) \in E\},$$
with 
$$A_{(v_i, v_j, \sigma)} = \left ( e_j e_i^{\tran}\right) \kron A_{\sigma},$$
where $e_i$ is the $i$th element of the canonical basis of $\reels^{\card{V}}$, and the operator "$\kron$" is the \emph{Kroenecker product} 
\footnote{
Given two matrices $A \in \reels^{m \times n}$ and $B \in \reels^{p \times q}$, 
the Kroenecker product of $A$ and $B$ is a matrix of dimension $(mp) \times (nq)$ defined as follows:
$$ A \kron B = \begin{bmatrix}
[A]_{1,1} B, & \cdots, & [A]_{1,n}B\\
\vdots, & \ddots, & \vdots\\
[A]_{m,1} B, & \cdots, & [A]_{m,n}B 
\end{bmatrix}, $$
where $[A]_{i,j}$ is the element on the ith row and jth column of $A$.
}
\label{def:lift}
between matrices.
\end{definition}
The lifting procedure allows to study the stability and boundedness properties
 of constrained switching systems through an equivalent, but higher dimensional, arbitrary switching system.
In \cite{KoTBWF}, Kozyakin introduced this lift to study constrained switching systems
 where the switching sequence is a realization of a Markov Chain. The results of \cite{KoTBWF} show the equality between the CJSR of these Markovian systems and the JSR of the arbitrary switching system on the lifted set of matrices.
 A similar lifting procedure was introduced in parallel by Wang et al. \cite{WaYuSOLA}
  for the general class of systems described by an automaton $\graph$.   
   Wang et al. proved that the stability of a constrained switching system is equivalent to that of the lifted set of matrices presented above. 
    From these results and those of Dai's \cite{DaAGTS}, it is easy to conclude the following: the CJSR of a constrained switching system $S(\graph, \mset)$ equals the JSR of the lifted matrix set $\mset_\graph$,
     and the boundedness of both the constrained and lifted arbitrary switching systems are equivalent.

This leads us to a very natural idea. Instead of devising complex boundedness conditions for constrained switching systems $S(\graph, \mset)$ with $\jsr(S) = 1$, we may just consider the irreducibility of the matrix set $\mset_\graph$, which would directly imply the boundedness of $S(\graph, \mset)$.\\
However, when analysing the system through its lift $\mset_\graph$, one loses all the combinatorial aspect of the dynamics encapsulated in the automaton. As illustrated in the next example, this leads to the fact that the irreducibility of a set $\mset_\graph$ is weaker than the condition of Theorem \ref{thm:non_def2}.

\begin{example}[Example  \ref{example:weakness}, cont.]
Consider the constrained switching system $S(\graph(V,E), \mset)$ of Example \ref{example:weakness}. 
For this system, there are 6 matrices in the set $\mset_\graph$, one per edge of the graph $\graph$ of Figure \ref{fig:ex_failing}. For constructing the set $\mset_\graph$, we let the nodes $v_1$, $v_2$ and $v_3$ be respectively the left, middle and right nodes in  Figure \ref{fig:ex_failing}. Then, we apply Definition \ref{def:lift} to build the matrices. For example, considering the edge $(v_1, v_2, 2) \in E$, the corresponding matrix is given by
$$A_{(v_1, v_2, 2)} = \left ( \begin{pmatrix} 0 \\ 1 \\0 \end{pmatrix} \cdot  \begin{pmatrix} 1 \\ 0 \\0 \end{pmatrix}^\top \right)  \kron A_2 = \begin{pmatrix} 0 & 0 & 0 \\ A_2 & 0 & 0 \\ 0 & 0& 0 \end{pmatrix}.
 $$
After constructing the set $\mset_\graph$, is is easy to verify that the subspace of $\reels^{3\cdot 2}$ defined as
$$ \subs{X} = \{ (x, y, z)^\top \in \reels^{3\cdot 2}, \, x = \in \reels^2, \, y \in \im{A_2}, \, z \in \im{A_2}  \},$$
is invariant under all matrices in $\mset_\graph$. Its is non-trivial with dimension 4 (recall that $A_2$ has rank 1 here).
We conclude that for the system $S(\graph, \mset)$ of Example \ref{example:weakness}, the lifted set $\mset_\graph$ is reducible. However, the system satisfies to the conditions of Theorem \ref{thm:non_def2}.
\end{example}

The following further qualifies the relation between Theorem \ref{thm:non_def2} and the irreducibility of $\mset_\graph$.
\begin{theorem}
Consider a system $S(\graph, \mset)$. If the lifted set of matrices 
$\mset_\graph$ is \emph{irreducible}, then for all $v \in V$, $v$ is irreducible.
\label{thm:liftisweak}
\end{theorem}
\begin{proof}

Assume that there exists one node $v' \in V$ which is reducible, with a  non-trivial invariant subspace $\subs{Y}$.
At each node $v \in V$, $v \neq v'$, we define the following subspace,
$$ \subs{X}_{v} = \{x \in \reels^n: \, \forall p \in E_{v, v'}^*, A_p x \in \subs{Y}\}.$$
and we let $\subs{X}_{v'} = \subs{Y}$. The fact that $v'$ is reducible means that each one of these subspaces is invariant at its corresponding node. They  may however be non-trivial.\\
From these subspaces we then construct a non-trivial invariant subspace $\subs{X} \subset \reels^{n\card{V}}$ of the set $\mset_\graph$ by taking the \emph{cartesian product} of the $\{\subs{X}_{v}\}_{v \in V}$. \\
This shows that if at least one node in $V$ is reducible, then the set $\mset_\graph$ is reducible as well.
\end{proof}

\section{A polynomial time algorithm for deciding dead-beat stability and CJSR = 0}
\label{sec:0cjsr}
We now turn our focus towards our second goal, 
that is to provide a polynomial time algorithm for deciding, given a system $S(\graph, \mset)$, 
whether $\jsr(S) = 0$ or $\jsr(S) > 0$.
 We prove that, as it is the case for arbitrary switching systems, $\jsr(S) = 0$ is equivalent to dead-beat stability. 
  The interest lies in that it can provide conditions, checked in polynomial time, 
  of dead-beat stabilizability
  for switched autonomous systems, an issue that has not been extensively addressed in the literature and no clear results have been provided so far.
   The paper \cite{CoStSDRA} gives a good introduction for that problem but the systems under consideration are non autonomous. The work \cite{FiacMill13a} deals with dead-beat stability in the context of state reconstruction for LPV systems. \\
  Our theoretical results are presented in Subsection~\ref{ssec_genframe}. 
   Then, they are illustrated in Subsection~\ref{ssec_electric}. 
   The issue of left inverter design is discussed for a real-world application. 

\subsection{General results}
\label{ssec_genframe}

Our main theoretical result in this section
 is the generalization of Theorem \ref{thm:0JSR}.

\begin{theorem}
Given a system $S(\graph, \mset)$, $\jsr(S) = 0$ if and only if for all paths $p$ of length $n\card{V}$, $A_p = 0$.
\label{thm:0CJSR}
\end{theorem}
\begin{proof}
The proof is direct from the relation between the CJSR of a system $S(\graph, \mset)$ and the JSR of the lift $\mset_\graph$ (Definition \ref{def:lift}). 
From Theorem \ref{thm:0JSR},
 we know that the JSR of the lifted system equals zero if and only if all products of length $n\card{V}$ of matrices in $\mset_\graph$ are equal to zero.
This is equivalent to asking that for all paths $p$ of length $n\card{V}$, $A_p = 0$. 
\end{proof}

The result above naturally indicates that $\jsr(S) = 0$ is decidable.
 This was of course known for the case of arbitrary switching systems, 
 and  L. Gurvits  \cite{GuSOLIP2} proposed (without proof) a polynomial time algorithm for checking whether the joint spectral radius of a set of matrices is zero or not. A proof is presented in \cite{JuTJSR}, Section 2.3.1.
\begin{proposition}[\emph{Gurvits' iteration \cite{GuSOLIP2}}]
Let $\mset \subset \reels^{n \times n}$ be a finite set of matrices, and let $U_0 = I$ be the identity matrix in $\reels^{n \times n}$.
For $k = 1, 2 , \ldots$, let 
$$U_k = \sum_{A \in \mset} A^\top U_{k-1} A. $$
The joint spectral radius of $\mset$ is zero if and only if $U_n = 0$.
\label{prop:Gurvits1}
\end{proposition}
The result above can be extended to constrained switching systems with the (expected) cost of 
taking the automaton $\graph$ into account. A first approach would be to apply Proposition \ref{prop:Gurvits1} 
to the lifted arbitrary switching system on the set $\mset_\graph$ obtained from Definition \ref{def:lift}. 
However, this requires the computation of $\mset_\graph$, and this first approach does not benefit of the combinatorial 
structure of the system $S(\graph, \mset)$.
Instead of doing this, we propose a second, more efficient, approach.
\begin{proposition}[\emph{Generalized Gurvits' iteration}]
Consider a system $S(\graph(V,E), \mset)$. Define, for all nodes $v \in V$, $U_0^v = I$, the identity matrix in $\reels^{n \times n}$.
For $k = 1,2, \ldots$, for all nodes $v$, let 
$$ U_k^v = \sum_{(v, w, \sigma) \in E_{v}^1} A_{\sigma}^\top U_{k-1}^{w} A_\sigma.$$
The system has $\jsr(\graph, \mset) = 0$  if and only if $U_{n\card{V}}^v = 0$ for all $v \in V$.
\label{prop:Gurvits2}
\end{proposition}
 \begin{proof}
 By construction, we obtain
 $$ U^v_k = \sum_{p \in E_{v}^k} A^\top_p A_p.$$
 These matrices are sum of positive semi-definite matrices. They equal zero if and only if all the matrices in the sum are zero.
Therefore, we have $U^v_{n \card{V}} = 0$ if an only if for all paths $p \in E_v^{n \card{V}}$, $A_p = 0$. The conclusion follows immediately.
 \end{proof}

Ending this subsection, we relate the case $\jsr(S) = 0$ with our main result in Section \ref{sec:bound}, Theorem~\ref{thm:non_def2}.
 
\begin{proposition}
Consider a  system $S(\graph(V,E), \mset)$ linearly connected and with an unavoidable set of irreducible nodes.
For this system, $\jsr(S) > 0$ and there exists a constant $K > 0$ such that for any $t \geq 0$ and any switching sequence $\sigma(0), \ldots, \sigma(t-1)$ accepted by $\graph$,
\begin{equation}
 \|A_\sigma(t-1) \cdots A_\sigma(0)\| \leq K \jsr(S)^t. 
\label{eq:non-def}
\end{equation}
\label{prop:non-def}
\end{proposition}
\begin{proof}
The proof that $\jsr(S) > 0$ is extracted from Theorem \ref{thm:bound} and from  \cite{JuTJSR} - Lemma 2.2. 
This second Lemma states that \emph{the joint spectral radius} of an \emph{irreducible set of matrices} is greater than 0.
\\
Now, take any irreducible node $v$ of the system.
From Theorem \ref{thm:bound}, the set of matrices 
$$\mset_v = \{A_c : c \in E_{v,v}^*,\,  \card{c} \leq 1 + n(\card{V}-1)\}, $$
is irreducible. 
From Lemma 2.2 in \cite{JuTJSR}  the joint spectral radius $\rho$ of this set is greater than zero. 
From this point, it is easy to conclude that  $\jsr(S) \geq \rho^{1/ \left ( 1 + n(\card{V}-1) \right ) }> 0$.\\
For the second part of the proposition, remark that the CJSR (\ref{eq:CJSR}) is homogeneous in the set of matrices $\mset$.
That is, if one scales all the matrices in $\mset$ by a same constant $\tau \geq 0$, the CJSR of the scaled system is then equal to $\tau \jsr(S)$.
Therefore, by scaling the set $\mset$ of the system $S$ by the positive constant $1/\jsr(S)$, the scaled system has a CJSR of 1. 
This scaled system then satisfies to the hypothesis of Theorem \ref{thm:non_def2},
 and by the definition of boundedness (\ref{eq:boundedness}) and homogeneity of the norm, we obtain
$$ \jsr(S)^{-t}\|A_{\sigma(t-1)} \cdots 
A_{\sigma(0)} \| \leq K, $$
for some constant $K$, for all $t \geq 0$ and all  switching sequences $\sigma(0), \ldots,\sigma(t-1)$ of the original system $S$. 
\end{proof}
We now move onto an example to show the applicability and efficiency of the framework. The example proposed below corresponds to a real-world case study.

\subsection{Example: designing a left inverter for an electrical vehicle}
\label{ssec_electric}

The vehicle under consideration has been deeply studied in \cite{MaCODU}. Actually, it is the prototype developed by the Research Center for Automatic Control of Nancy which is annually involved in the European Shell Eco-Marathon
race  in the Plug-in (battery) category.
The nonlinear discrete-time model with state variables $x^{(1)}_t$ and $x^{(2)}_t$ denoting respectively the position and the velocity at time $t$ admits the state space form
\begin{equation}
\begin{array}{lcl}
x_{t+1} & = & \begin{pmatrix}
			1&a(T_t)\\
			0&b(T_t, x^{(2)}_t)\\
		\end{pmatrix}x_t+\begin{pmatrix}		
			0\\ c(T_t)\\
		\end{pmatrix}u_t
\end{array}
\label{eq_syst}
\end{equation}
with time-varying entries 
$ a(T_t)$, $b(T_t, x^{(2)}_t)$ and $c(T_t)$ as a function of $T_t$ and/or $x^{(2)}_t$ and of physical parameters (mass, dimensions, aerodynamics, \ldots) of the vehicle.
	The quantity $T_t$ is the \emph{sampling period} that may be either constant or time-varying. We denote with $\rho_t$ the vector $\rho_t=\big(a(T_t),b(T_t, x^{(2)}_t),c(T_t)\big)$.\\
Let us consider the position $x_t^{(1)}$ as the measured output $y_t$, that is $y_t=\begin{pmatrix}
			1&0
		\end{pmatrix}x_t.	$
 Hence, the system \eqref{eq_syst} can be expressed as
\begin{equation}
\begin{aligned}
\label{eq_LPV}
x_{t+1}&=A(\rho_t)x_t+B(\rho_t)u_t, \\
y_t & = C x_t.
\end{aligned}
\end{equation}
We aim at designing a \emph{left inverter} for \eqref{eq_syst}.
A left inverter is a dynamical system  able to give, 
for every time $t$, the input $u_t$ from
 a sequence of measures $y_t,y_{t-1},\ldots$. 
 This left inverter has a practical interest insofar 
 as it will allow to estimate the control $u_t$ from the
  measures $y_t,y_{t-1},\ldots$ delivered by the sensors. 
  When placed on-board, it will be involved in an
   actuator default detection module 
   that plays a central role for the sake of safety. \\
Left inversion for discrete-time linear switched 
(and by extension Parameter Varying) 
systems has been addressed in several papers
 for different purposes \cite{SuHaDSIA,MiDaIAFO,MiDaFOSL,PaMiNSFT}. From those papers, it is recalled that 
the stability of the left inverter,
 that is the convergence in finite time of the input estimate $\hat{u}_t$ delivered by the left inverter to the actual input $u_t$ is 
 guaranteed whenever the auxiliary system defined below is dead-beat stable. 
\begin{equation}
{\epsilon}_{t+1} = Q(\rho_t,\ldots,\rho_{t+r}) {\epsilon}_t,
\label{eq:INVSYS}
\end{equation}
where the integer $r$ corresponds to the relative degree of the system. The relative degree is the minimum number of iterations $r$ such that the output $y_{t+r}$ depends explicitely on the input $u_t$ \cite{PaMiNSFT}. \\ 
The matrix $Q(\rho_t,\ldots,\rho_{t+r})$ describes the dynamics of the reconstruction error  $\epsilon_t=u_t-\hat{u}_{t}$.
For the system \eqref{eq_LPV}, the relative degree is $r=2$ since the product $CA(\rho_{t+1})B(\rho_t)\neq 0$, and $Q(\rho_t, \rho_{t+1}, \rho_{t+2})$ is given by
$$
Q(\rho_t,\rho_{t+1},\rho_{t+2})=A(\rho_t)-B(\rho_t)\big(CA(\rho_{t+1})B(\rho_t)\big)^{-1}CA(\rho_{t+1})A(\rho_{t}).
$$
It turns out that in our case, $Q({\rho_t,\rho_{t+1},\rho_{t+2}})$ depends only on the sampling periods at times $t$ and $t+1$, that is
		$$
		Q({\rho_t,\rho_{t+1},\rho_{t+2}})=
			\begin{pmatrix}
				1&a(T_t)\\
				-a(T_{t+1})^{-1}&-a(T_{t+1})^{-1}a(T_t)\\
			\end{pmatrix}=Q({T_t,T_{t+1}}).
		$$
In this case study, the sampling period $T_t$ takes  $2$ possible values $T^1$ or $T^2$ according to the range of variation of the velocity of the vehicle. The faster the vehicle, the lower the sampling period.  Then, we can rewrite the auxiliary system \eqref{eq:INVSYS} as
\begin{equation}
\label{eq_aux2}
{\epsilon}_{t+1} = Q_{\sigma(t)} {\epsilon}_t  ,
\end{equation}
which takes the form of a switched linear system that assigns to each sequence
 $(T_t,T_{t+1})$ starting at time $t$ a mode $\sigma(t) \in \{1,2,3,4\}$. 
 This explicitly induces constraints on the the switching sequences. 
 Indeed, let us consider that at a given time $t$, the mode $\sigma(t)$ corresponds to $(T_t,T_{t+1})=(T_t,T^2)$. 
 At time $t+1$, we need a mode $\sigma(t+1)$ corresponding to a sequence of the form $(T_{t+1}, T_{t+2})=(T^2,T_{t+2})$. 
 Therefore, the auxiliary system can be expressed as a system $S(\graph, \mset)$ on $N = 4$ modes. 
 A valid automaton for the system is given in Figure \ref{fig:graphFlat}.
 \begin{figure}[!ht]
\centering
\includegraphics[scale=0.5]{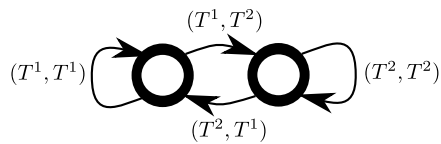}
\caption{Every mode $\sigma(t)$ of the inverse system corresponds to a sequence $(T_t, T_{t+1}) \in \{T^1, T^2\}^2$ 
of the original system. There are 4 of such pairs, and 4 possible transitions between these pairs.}
\label{fig:graphFlat}
\end{figure}

For our system,  Theorem \ref{thm:0CJSR} is verified and
 the algorithm of Proposition \ref{prop:Gurvits2} 
 confirms that $\jsr(S) = 0$.
This  can be verified by hand, noticing that the product $Q({T_{t+1},T_{t+2}})Q({T_t,T_{t+1}})=0$, for all $t \geq 0$ when following the switching rules. 
  The auxiliary system is here \emph{dead-beat} stable. 
  It is worth noticing that, since the product $ Q({T_t,T_{t+1}})Q({T_{t+1},T_{t+2}})\neq 0$, the \emph{arbitrary switching system} on the set of matrices defining the auxiliary system is not deadbeat stable.
   This highlights again that the result is not trivial, and that a tool adapted for constrained switching systems was required.

\section{Conclusion}
Discrete-time linear switching systems with constrained switching sequences are a rich class of dynamical systems. 
They present many challenging problems regarding their stability and boundedness properties.\\ 
In this work, we focused on two aspects. 
First, we provided decidable sufficient boundedness conditions for the case when the constrained joint spectral radius (\ref{eq:CJSR}) of a  system  is equal to one. 
Second, we provided a polynomial time algorithm for deciding when all trajectories of a system vanish to the origin in finite time.\\
For future research direction, we first point out that while we proved the decidability of our boundedness conditions, 
we currently have no \emph{polynomial time} algorithm for their verification (see Remark \ref{rem:polytime}). 
To the best of our knowledge, this problem is open even in the special case of arbitrary switching systems, 
for which the conditions boil down to the irreducibility of a set of matrices (Definition \ref{def:red}). 
Second, it has been shown in \cite{PaMiNSFT} that deciding dead-beat stability can be used for flatness analysis of switching linear systems.
 Hence, we expect that the algorithm presented here will lead to interesting alternatives for characterizing the flatness of switching systems.
The study of the stability properties of constrained switching systems arising as left inverters (see Subsection \ref{ssec_electric}), or by extension, Unknown Input Observers, also deserves further inspection.

\section{Bibliography}
\label{sec:biblio}
\bibliographystyle{abbrv}
\bibliography{biblio} 

\end{document}